\documentclass[12pt,a4paper]{amsart}
\usepackage[all]{xy}

\usepackage{diagrams,amssymb}

\newcommand{\newsect}[1]{\vskip7mm\section{#1}}           

\diagramstyle[height=2em,width=2em,midshaft,labelstyle=\scriptstyle]
\newarrow{To}----{->}
\newarrow{Epi}----{->>}
\newarrow{Mono}>---{->}
\newarrow{Iso}>---{->>}
\newarrow{Mapsto}|---{->}
\newarrow{Igual}=====
\newarrow{Dashto}{}{dash}{}{dash}{->}

\newcommand{\Z}{{\mathbb Z}}
\newcommand{\Q}{{\mathbb Q}}

\newcommand{\F}{{\mathbb F}}

\newcommand{\pcom}{{}_{p}^{\wedge}}

\DeclareMathAlphabet\EuR{U}{eur}{m}{n}
\SetMathAlphabet\EuR{bold}{U}{eur}{b}{n}

\newcommand{\Inj}{\operatorname{Inj}\nolimits}

\newcommand{\defeq}{\overset{\text{\textup{def}}}{=}}

\renewcommand{\:}{\colon}

\newcommand{\calb}{\mathcal{B}}

\newcommand{\calf}{\mathcal{F}}

\newcommand{\call}{\mathcal{L}}

\newcommand{\calp}{\mathcal{P}}

\newcommand{\curs}{\EuR}

\newcommand{\nsg}{\vartriangleleft}
\newcommand{\Id}{\operatorname{Id}\nolimits}
\newcommand{\incl}{\operatorname{incl}\nolimits}

\newcommand{\Inn}{\operatorname{Inn}\nolimits}

\newcommand{\isotyp}{_{\textup{typ}}}      

\let\oldcirc=\circ
\renewcommand{\circ}{\mathchoice
    {\mathbin{\scriptstyle\oldcirc}}{\mathbin{\scriptstyle\oldcirc}}
    {\mathbin{\scriptscriptstyle\oldcirc}}
    {\mathbin{\scriptscriptstyle\oldcirc}}}

\newcommand{\hclim}[1]{\setbox1=\hbox{\rm hocolim}
    \setbox2=\hbox to \wd1{\rightarrowfill} \ht2=0pt \dp2=-1pt
    \mathop{\vtop{\baselineskip=5pt\box1\box2}}
    _{#1}}

\newcommand{\higherlim}[2]{\displaystyle\setbox1=\hbox{\rm lim}
	\setbox2=\hbox to \wd1{\leftarrowfill} \ht2=0pt \dp2=-1pt
	\setbox3=\hbox{$\scriptstyle{#1}$}
	\ifdim\wd1<\wd3
	\mathop{\hphantom{^{#2}}\vtop{\baselineskip=5pt\box1\box2}^{#2}}_{#1}
	\else
	\mathop{\vtop{\baselineskip=5pt\box1\box2}}\limits_{#1}\nolimits^{#2}
	\fi}

\renewcommand{\hom}{\operatorname{Hom}\nolimits}
                     \newcommand{\Hom}{\operatorname{Hom}\nolimits}

                     \newcommand{\Rep}{\operatorname{Rep}\nolimits}
\newcommand{\Iso}{\operatorname{Iso}\nolimits}
\newcommand{\Aut}{\operatorname{Aut}\nolimits}
\newcommand{\Out}{\operatorname{Out}\nolimits}
\newcommand{\End}{\operatorname{End}\nolimits}

\newcommand{\Mor}{\operatorname{Mor}\nolimits}
\newcommand{\Ker}{\operatorname{Ker}\nolimits}

\newcommand{\rk}{\operatorname{rk}\nolimits}
\newcommand{\xxto}[1]{\mathrel{\mathop{%
  \setbox0\hbox{$\ {\scriptstyle#1}\ $}%
  \hbox to \wd0{\rightarrowfill}}^{#1}}%
}

\newcommand{\xto}[2][]{%
  \mathrel{\mathop{%
    \setbox0\vbox{
      \hbox{$\scriptstyle\;\;{#1}\;\;$}%
      \hbox{$\scriptstyle\;\;{#2}\;\;$}%
    }%
    \hbox to\wd0{\rightarrowfill}\displaystyle}%
  \limits^{#2}\ifx{#1}{}\else{_{#1}}\fi}%
}

\newcommand{\longleft}[1]{\;{\leftarrow%
\count255=0 \loop \mathrel{\mkern-6mu}%
    \relbar\advance\count255 by1\ifnum\count255<#1\repeat}\;}
\newcommand{\longright}[1]{\;{\count255=0 \loop \relbar\mathrel{\mkern-6mu}%
    \advance\count255 by1\ifnum\count255<#1\repeat\rightarrow}\;}
\newcommand{\Right}[2]{\overset{#2}{\longright#1}}
\newcommand{\RIGHT}[3]{\mathrel{\mathop{\kern0pt\longright#1}
        \limits^{#2}_{#3}}}

\newcommand{\LEFT}[3]{\mathrel{\mathop{\kern0pt\longleft#1}\limits^{#2}_{#3}}
}
\newcommand{\dRIGHT}[3]{\mathrel{%
   \mathop{\vcenter{\baselineskip=0pt\hbox{$\kern0pt\longright#1$}%
   \hbox{$\kern0pt\longright#1$}}}\limits^{#2}_{#3}}}
\newcommand{\LRIGHT}[3]{\mathrel{%
   \mathop{\vcenter{\baselineskip=0pt\hbox{$\kern0pt\longleft#1$}%
   \hbox{$\kern0pt\longright#1$}}}\limits^{#2}_{#3}}}
\newcommand{\RLEFT}[3]{\mathrel{%
   \mathop{\vcenter{\baselineskip=0pt\hbox{$\kern0pt\longright#1$}%
   \hbox{$\kern0pt\longleft#1$}}}\limits^{#2}_{#3}}}
\newcommand{\onto}[1]{\;{\count255=0 \loop \relbar\joinrel
    \advance\count255 by1
    \ifnum\count255<#1 \repeat \twoheadrightarrow}\;}

\newtheoremstyle{slant}{}{}{\slshape}{}{\bfseries}{.}{.5em}{}%
\newtheoremstyle{special}{}{}{\slshape}{}{\bfseries}{.}{.5em}{\thmnote{#3}}

\newtheorem{Thm}{Theorem}[section]
\newtheorem{Prop}[Thm]{Proposition}
\newtheorem{Cor}[Thm]{Corollary}
\newtheorem{Lem}[Thm]{Lemma}

\newtheorem{Th}{Theorem}

\theoremstyle{definition}
\newtheorem{Rmk}[Thm]{Remark}

\newtheorem{Defi}[Thm]{Definition}

\theoremstyle{remark}

\newcommand{\SFL}{(S,\calf,\call)}

\newcommand{\units}[1]{\Gamma_{#1}(p)}
\newcommand{\ord}{\mathrm{ord}}
\newcommand{\calh}{\mathcal{H}}
\newcommand{\calg}{\mathcal{G}}

\newcommand{\mset}{\curs{m}}
\newcommand{\objset}{\curs{P}}
\newcommand{\adams}{\curs{Ad}}

\newcommand{\A}{\curs{a}}
\newcommand{\homf}{\Hom_\calf}
\newcommand{\isof}{\Iso_\calf}

\newcommand{\N}{\mathbb{N}}
\newcommand{\Der}{\mathrm{Der}}
\newcommand{\PDer}{\mathrm{PDer}}
\newcommand{\n}{\curs{n}}
\newcommand{\morf}[2]{\Hom_\calf(#1,#2)}
\newcommand{\morl}[2]{\Mor_\call(#1,#2)}

\def\al{\alpha}
\def\be{\beta}
\def\de{\delta}
\def\ga{\gamma}
\def\Ga{\Gamma}
\def\la{\lambda}
\def\si{\sigma}
\def\vp{\varphi}

\def\F{\mathcal{F}}
\def\L{\mathcal{L}}

\newcommand{\cset}{\curs{c}}
\newcommand{\dset}{\curs{d}}
\newcommand{\iset}{\curs{i}}
\newcommand{\Mset}{\curs{M}} \def\tMset{\widetilde{\Mset}}
\newcommand{\nset}{\curs{n}}
\newcommand{\rset}{\curs{r}}

\def\GL{\operatorname{GL}}
\def\res{\operatorname{res}}

\theoremstyle{definition}

\renewcommand{\labelenumi}{\textup{(\roman{enumi})}}%

\renewenvironment{enumerate}{\begin{list}%
{\labelenumi}
{\usecounter{enumi}%
\setlength{\itemindent}{0pt}%
\settowidth{\labelwidth}{\labelenumi}%
\addtolength{\labelwidth}{\labelsep}%
\setlength{\leftmargin}{\labelsep}%
\addtolength{\leftmargin}{\labelwidth}%
\setlength{\listparindent}{0pt}%
\setlength{\itemsep}{6pt}%
\setlength{\parsep}{0pt}%
\setlength{\topsep}{6pt}%
}}{\end{list}}

\title{Unstable Adams Operations on $p$-local Compact Groups}
\author{Fabien Junod}
\address{Procter and Gamble}
\email{junod.f@pg.com}
\author{Ran Levi}
\address{Institute of Mathematics, University of Aberdeen,
Fraser Noble Building, Aberdeen AB24 3UE, U.K.}
\email{r.levi@abdn.ac.uk}
\author{Assaf Libman}
\address{Institute of Mathematics, University of Aberdeen,
Fraser Noble Building, Aberdeen AB24 3UE, U.K.}
\email{a.libman@abdn.ac.uk}


\begin{document}

\begin{abstract}
A $p$-local compact group is an algebraic object modelled on the $p$-local homotopy theory of classifying spaces of compact Lie groups and $p$-compact groups. In the study of these objects unstable Adams operations,  are of fundamental importance. In this paper we define unstable Adams operations within the theory of $p$-local compact groups, and show that such operations exist under rather mild conditions. More precisely, we prove that for a given $p$-local compact group $\calg$ and a sufficiently large positive integer $m$, there exists an injective group homomorphism from the group of $p$-adic units which are congruent to 1 modulo $p^m$ to the group of unstable  Adams operations on $\calg$. \end{abstract}

\maketitle

Let $p$ be a prime number. 
In \cite{BLO3} Broto, the second author and Oliver developed the theory of $p$-local compact groups. 
The theory is modelled on the $p$-local homotopy theory of classifying spaces of compact Lie groups and $p$-compact groups, and generalises the earlier concept of $p$-local finite groups \cite{BLO2}. 
It provides a coherent context in which classifying spaces of compact Lie groups and $p$-compact groups \cite{DW} can be studied, and also gives rise to many exotic examples. 
Roughly speaking, a $p$-local compact group is a triple $\calg=\SFL$, where $S$ is a discrete $p$-toral group, $\F$ is a saturated fusion system over $S$, and $\L$ is a centric linking system associated to $\calf$.
More specifically, the group $S$ is  an extension of a finite $p$-group by a group of the form $(\Z/p^\infty)^r$, where $\Z/p^\infty = \cup_k \Z/p^k$. One can think of $S$ as a Sylow $p$-subgroup of $\calg$ in an appropriate sense. 
The saturated fusion system $\calf$ is a category whose set of objects consists of all subgroups of $S$, and whose morphisms model conjugacy relations among subgroups of $S$. The centric linking system  $\call$ is again a category which is an enrichment of $\calf$ with just about enough structure to allow one to associated with $\calf$ a homotopy theoretically meaningful "classifying space"  given by the space $|\call|\pcom$ which we shall denote by $B\calg$. Here $(-)\pcom$ denotes the Bousfield-Kan \cite{BK} $p$-completion functor, and $|-|$ stands for the functor which associates with a small category its nerve.  
Compact Lie groups and $p$-compact groups give rise to $p$-local compact groups, whose classifying spaces coincide up to homotopy with the $p$-completed classifying space of the object from which they originate. 

Unstable Adams operations are certain self equivalences of classifying spaces of compact Lie groups, which were also defined and studied for $p$-compact groups (see for instance \cite{JMO, AGMV}). 
For a compact connected Lie group $G$, one defines an unstable Adams operation of degree $k$ on $G$, for some integer $k$, to be any self map of the classifying space $BG$ which induces multiplication by $k^i$ on $H^{2i}(BG,\Q)$. 
There is an analogous definition for connected $p$-compact groups, where the degree is allowed to be any $p$-adic unit, and rational cohomology is replaced by $H^*_{\Q_p} (-) = H^*(-,\Z\pcom)\otimes \Q$. 

Unstable Adams operations are of fundamental importance in the study of classifying spaces of compact Lie and $p$-compact groups, which motivates the study undertaken in this paper, the aim of which is to prove the existence unstable Adams operations on $p$-local compact groups. 

For a compact Lie group $G$, the rational cohomology $H^*(BG,\Q)$ is given by the invariants of the action of the Weyl group of $G$ on $H^*(BT,\Q)$, where $T\le G$ is a maximal torus \cite{Bo}. Thus it is not hard to see, as we will observe in Section \ref{sec:adams-automorphisms}, that the definition of an unstable Adams operation of degree $k$ on $BG$ is equivalent to the existence of a self map of $BG$, which extends the map induced by the $k$-th power map on $BT$. The analogous description of cohomology holds for connected $p$-compact groups \cite[Thm. 9.7]{DW}, and so here as well one can define an unstable Adams operation as a self map extending the map induced by the appropriate power map on the maximal torus.

A precise algebraic definition of what unstable Adams operations actually mean for $p$-local compact groups will be given in Section \ref{sec:unstable-adams-operations}. Roughly and more geometrically speaking, given a $p$-local compact group $\calg$ and a $p$-adic unit $\zeta$, an unstable Adams operation of $\calg$ of degree $\zeta$ is a self equivalence  of $B\calg$, which restricts up to homotopy to the map induced by the $\zeta$ power map on the classifying space of the maximal torus of $\calg$ (see Definition \ref{dpt-groups}). A comparison of this notion with the algebraic concept will be carried out in Section \ref{sec:unstable-adams-operations} as well.
The collection of all unstable Adams operation on $\calg$ forms a group under composition called the group of unstable Adams operations on $\calg$, which we denote by $\adams(\calg)$.

For a positive integer $k$, let $\units{k}$ denote the subgroup  of all  $p$-adic units $\zeta\in\Z_p^\times$ such that $\zeta-1$ is divisible by $p^k$.
Thus $\units{0} = \Z_p^\times$, and for each $k>0$, elements of $\units{k}$ have the form $1+a_kp^k+\cdots$. 
We are now ready to state our main theorem.

\begin{Th}\label{Main-intro}
Let $\calg$ be a $p$-local compact group. 
Then for a sufficiently large positive integer $k$ there exists a group homomorphism
\[
\A_\calg\colon \units{k} \to \adams(\calg),
\]
such that for any $\zeta\in\units{k}$, $\A_\calg(\zeta)$ is an unstable  Adams operation of $\calg$  of degree $\zeta$.
\end{Th}

The theorem will be restated as Theorem \ref{main-thm} in terms both more precise and more suitable to the theory of $p$-local compact groups. 
We will not deal in this paper with the question of finding a lower bound for the integer $k$ in the theorem, nor will we attempt to answer the obvious question of uniqueness up to homotopy, once existence is established. 

We point out that in our construction no connectivity assumption has to be made, although one would expect that some conditions, analogous to connectivity in compact Lie groups, will have to be assumed in dealing with uniqueness questions.

This paper is an elaboration of the PhD dissertation of the first author.
It is organised as follows. 
Sections \ref{sec:plcg} collects some basic facts on $p$-local compact groups.
In Section \ref{sec:adams-automorphisms}  we define and study Adams automorphisms of $p$-toral groups, and in Section \ref{sec:unstable-adams-operations} we define and discuss unstable Adams operations on $p$-local compact groups. 
Finally, in Section \ref{sec:main-construction} we prove Theorem \ref{Main-intro}.
More precisely, for a $p$-local compact group $\calg$, we construct in Theorerm \ref{main-thm} families of subsgroups of the group  $\adams(\calg)$ of unstable Adams operations of $\calg$.

The authors are grateful to Alex Gonzalez for reading several versions of their manuscript and making useful suggestions.

\newsect{Preliminaries on $p$-local Compact Groups.}\label{sec:plcg}
We recall the definition and some basic properties of $p$-local compact groups. The reader is referred to \cite{BLO3} for a comprehensive account of these objects. 

The fundamental objects in this theory are discrete $p$-toral groups. By $\Z/p^\infty$ we mean the union of all $\Z/p^r$ with respect to the natural inclusions. 

\begin{Defi}\label{dpt-groups}
A group of the form $(\Z/p^\infty)^r$ for some positive integer $r$ will be referred to as a \emph{discrete $p$-torus}. 
A \emph{discrete $p$-toral group} is an extension $S$ of a finite $p$-group $\Gamma$ by a discrete $p$-torus $S_0$. 
The normal subgroup $S_0$ will be referred to as the \emph{maximal torus} or the \emph{identity component} of $S$, and the quotient group $\Gamma\cong S/S_0$ will be called the \emph{group of components} of $S$.
\end{Defi}

The identity component $P_0$ of a discrete $p$-toral group $P$ can be characterised  as the subset of all infinitely $p$-divisible elements in $P$, and also as the minimal subgroup of finite index in $P$.  
Thus, $P_0$ is a characteristic subgroup.
The rank of $P$ is the number $r=\rk(P)$ such that $P_0\cong(\Z/p^\infty)^r$.

Recall that given $P,Q \leq S$ the elements $g \in S$ such that $gPg^{-1} \leq Q$ form the transporter set $N_S(P,Q)$.
The set $\Hom_S(P,Q)$ of all homomorphisms $c_g \colon P \to Q$ , which are restrictions of an inner automorphism of $S$ is obtained by identifying two elements in $N_S(P,Q)$ if they differ by an element of the centraliser $C_S(P)$. Let $\Inj(P,Q)$ denote the set of all the injective homomorphisms $P \to Q$.
We are now ready to recall the definition of  fusion systems over discrete $p$-toral groups.

\begin{Defi} \label{Frob.def.}
A \emph{fusion system} $\calf$ over a discrete $p$-toral group $S$ is a 
category whose objects are the subgroups of $S$, and whose morphism sets
$\homf(P,Q)$ satisfy the following conditions:
\begin{enumerate}
\item $\hom_S(P,Q)\subseteq\homf(P,Q)\subseteq\Inj(P,Q)$ for all $P,Q\le S$.

\item Every morphism in $\calf$ factors as an isomorphism in $\calf$
followed by an inclusion.
\end{enumerate}
\end{Defi}

Two subgroups $P,P'\le{}S$ are called \emph{$\calf$-conjugate} if $\isof(P,P')\ne\emptyset$. 
A subgroup $P\le S$ is said to be $\calf$-\emph{centric} if for every subgroup $P'\le S$ which is $\calf$-conjugate to $P$, $C_S(P')=Z(P')$.

The theory of $p$-local compact groups makes a very essential use of an extra set of axioms one requires fusion systems to satisfy. 
Fusion systems that satisfy these axioms are called \emph{saturated}. 
In this paper we do not directly need to use the saturation axioms, and we will therefore not spell them out. 
It should be emphasised, however, that throughout the paper only saturated fusion systems will be considered and we will heavily rely upon some of their properties.
The interested reader is referred to \cite[Def. 2.2]{BLO3}.

The next fundamental concept in the theory is that of a centric linking system. 

\begin{Defi}  \label{L-cat}
Let $\calf$ be a fusion system over a discrete $p$-toral group $S$.  A 
\emph{centric linking system associated to $\calf$} is a category $\call$ 
whose objects are the $\calf$-centric subgroups of $S$, together with a 
functor 
	$$ \pi \:\call\Right5{}\calf^c, $$
and ``distinguished'' monomorphisms $P\Right1{\delta_P}\Aut_{\call}(P)$ 
for each $\calf$-centric subgroup $P\le{}S$, which satisfy the following 
conditions.

\begin{enumerate}  
\renewcommand{\labelenumi}{\textup{(\Alph{enumi})}}%
\item  $\pi$ is the identity on objects and surjective on morphisms.
More precisely, for each pair of objects 
$P,Q\in\call$, $Z(P)$ acts freely on 
$\Mor_{\call}(P,Q)$ by composition (upon identifying $Z(P)$ with 
$\delta_P(Z(P))\le\Aut_{\call}(P)$), and $\pi$ induces a bijection
	$$ \Mor_{\call}(P,Q)/Z(P) \Right5{\cong} \homf(P,Q). $$

\item  For each $\calf$-centric subgroup $P\le{}S$ and each $g\in{}P$, 
$\pi$ sends $\delta_P(g)\in\Aut_{\call}(P)$ to 
$c_g\in\Aut_{\calf}(P)$.

\item  For each $f\in\Mor_{\call}(P,Q)$ and each $g\in{}P$, the following 
square commutes in $\call$:
\[\xymatrix{
P \ar[rr]^{f}\ar[d]_{\delta_P(g)} & &
Q \ar[d]^{\delta_Q(\pi(f)(g))}
\\ 
P \ar[rr]^{f} & &
Q.
}\]
\end{enumerate}

\end{Defi}

Next we recall the definition of our fundamental object. 

\begin{Defi} \label{p-lcg}
A \emph{$p$-local compact group} is a triple $\calg=\SFL$, where $S$ is 
a discrete $p$-toral group, $\calf$ is a saturated fusion system over $S$, and 
$\call$ is 
a centric linking system associated to $\calf$.  The \emph{classifying space} of 
$\calg$ is the $p$-completed nerve $|\call|\pcom$, which we will generally denote by $B\calg$.
\end{Defi}

In \cite{BLO3} the authors show that  compact Lie groups and $p$-compact groups give rise to particular examples of $p$-local compact groups. Another large family of examples arises from  linear torsion groups. In all cases the respective classifying space coincides up to homotopy (after $p$-completion in the case of genuine groups) with the classifying space of the $p$-local compact group it gives rise to.

Before moving on, we will need to record several key properties of linking systems that will be used later.

\begin{Prop}\label{prop-extend-hats-in-L}
There is a choice of functions $\de_{P,Q} \colon N_S(P,Q) \to \morl{P}{Q}$ for all $\calf$-centric $P,Q \leq S$ such that the following  holds.
\begin{itemize}
\item
For all $x \in N_S(P,Q)$, $\pi(\de_{P,Q}(x))=c_x$,

\item
The restriction of $\de_{P,P}\colon N_S(P) \to \Aut_\L(P)$  to $P$ is equal to the distinguished monomorphism $\de_P \colon P \to \Aut_\call(P)$, and 

\item
$\de_{Q,R}(x) \circ \de_{P,Q}(y) = \de_{P,R}(xy)$ for all $x \in N_S(Q,R)$ and all $y \in N_S(P,Q)$.
\end{itemize}
\end{Prop}

\begin{proof}
For every $P \in \calf^c$ choose once and for all morphisms $\iota_P \in \morl{P}{S}$ which project to the inclusion $P \leq S$ in $\F$. In particular let $\iota_S$ be the identity on $S$.
The proof is identical to that in \cite[Prop. 1.11]{BLO2}.
One uses \cite[Lemma 4.3(a)]{BLO3} instead of \cite[Prop. 1.10(a)]{BLO2}.
\end{proof}

\begin{Rmk}\label{remark-hats}
We will write $[\vp]$ for the image in $\calf$ of a morphism $\vp \in \call$ under the projection $\pi$.
Throughout this paper we will always assume that $p$-local compact groups are equipped with a choice of functions $\de_{P,Q}$ as in Proposition \ref{prop-extend-hats-in-L}.
The image of $g \in N_S(P,Q)$ under $\delta_{P,Q}$ will be denoted $\widehat{g}$, and with this notation $\widehat{g} \circ \widehat{h}=\widehat{gh}$.
For $P \leq Q$, the image of the identity element $e \in N_S(P,Q)$ will be denoted $\iota_P^Q$.
Note that $\iota_P^P=\Id_P$, and that  $[\widehat{g}]=c_g$.
\end{Rmk}

\begin{Lem}\label{lem-factor-in-L}
Let $\SFL$ be a $p$-local compact group.
\begin{enumerate}
\item
\label{lem-factor-in-L:mon}
Let $P \xto{a} Q \xto{b} R$ be morphisms in $\F^c$.
If $\be \in \morl{Q}{R}$ and $\ga \in \morl{P}{R}$ satisfy $[\be]=b$ and $[\ga]=b \circ a$ then there exists a unique 
$\al \in \morl{P}{Q}$ such that $[\al]=a$ and $\be \circ \al = \ga$.

\item
\label{lem-factor-in-L:epi}
If $P,Q,R$ and $a,b$ and $\ga$ as above and if $\al \in \morl{P}{Q}$ satisfies $[\al]=a$ then there exists a unique $\be \in \morl{Q}{R}$ such that $\be \circ \al = \ga$, (but $\beta$ is not generally a lift of $b$).

\item
\label{lem-factor-in-L:fac}
For any $\vp \in \morl{P}{Q}$ there exists a unique isomorphism $\vp' \in \Iso_\L(P,P')$ such that $\vp = \iota_{P'}^Q \circ \vp'$.

\end{enumerate}
\end{Lem}

\begin{proof}
Point \ref{lem-factor-in-L:mon} is proven in \cite[Proposition 4.3(a)]{BLO3}, and point \ref{lem-factor-in-L:fac} is an immediate consequence of \ref{lem-factor-in-L:mon}, since one has a factorization $[\vp]=\incl_{P'}^Q \circ f$, for some $f \in \Iso_\F(P,P')$.

It remains to prove point \ref{lem-factor-in-L:epi}.  We may assume, in light of point \ref{lem-factor-in-L:fac}, that $P \leq Q$ and that $\al=\iota_P^Q$.
Choose an arbitrary $\be' \in \morl{Q}{R}$ such that $[\be']=b$.
Then $[\ga]=b \circ a = [\be' \circ \iota_P^Q]$.
By axiom (A) in Definition \ref{L-cat} there is some $z \in Z(P)$ such that $\ga=\be' \circ \iota_P^Q \circ \widehat{z}=\be' \circ \widehat{z} \circ \iota_P^Q$.
Note that $z \in Z(P) \leq Q$ and therefore $\be\defeq\be' \circ \widehat{z} \in \morl{Q}{R}$ satisfies $\ga=\be \circ \iota_P^Q$.

To prove uniqueness, assume that $\be_1 \circ \iota_P^Q = \be_2 \circ \iota_P^Q$ for some $\be_1, \be_2 \in \morl{Q}{R}$.
In particular $[\be_1]|_P=[\be_2]|_P$, and \cite[Proposition 2.8]{BLO3} implies that there exists $g \in Z(P)$ such that $[\be_2]=[\be_1] \circ c_g$.
By axioms (A) and (B) of Definition \ref{L-cat} and by Remark \ref{remark-hats}, there is some $z \in Z(Q)$ such that $\be_2 = \be_1 \circ \widehat{gz}$.
Note that $Z(Q)=C_S(Q) \leq C_S(P)=Z(P)$ so $gz \in Z(P)$.
By the hypothesis on $\be_1$ and $\be_2$ we now get $\be_1 \circ \iota_P^Q =\be_2\circ \iota_P^Q = \be_1 \circ \widehat{gz} \circ \iota_P^Q = \be_1 \circ \iota_P^Q \circ \widehat{gz}$.
Since $Z(P)$ acts freely on $\morl{P}{R}$ we deduce that $gz=1$ and therefore $\be_1=\be_2$, whence the morphism $\be$ is unique.
\end{proof}

\begin{Cor}[{cf. \cite[Proposition 3.10]{BCGLO1}}]\label{cor-morphisms-of-L-are-mono-and-epi}
All morphisms in a centric linking system $\call$ associated to a saturated fusion system are both monomorphisms and epimorphisms in the categorical sense.
\end{Cor}

\begin{proof}
If $\be \circ \al_1=\be \circ \al_2$ in $\L$ then $[\al_1]=[\al_2]$, since $[\beta]$ is group monomorphism. Therefore $\alpha_1=\alpha_2$, by point \ref{lem-factor-in-L:mon} of Lemma \ref{lem-factor-in-L}. Hence all the morphisms in $\L$ are monomorphisms.
Similarly, by point \ref{lem-factor-in-L:epi} of the same Lemma, if $\be_1 \circ \al = \be_2 \circ \al$ in $\L$, then $\be_1=\be_2$, hence every morphism in $\L$ is an epimorphism.
\end{proof}

An important tool in the study of $p$-local compact groups, which will be useful in this paper as well, is the reduction of the object set of $\calf$ to a subset of subgroups of $S$ which contains only a finite number of $S$ conjugacy classes. The precise details of this construction are only relevant in the proof of Proposition \ref{fusion-preserving-respect-bullet} and are recalled there. 
It suffices to say that with any $P\le S$ one associates a subgroup $P^\bullet$ containing $P$. 
The family of objects which results from this construction has some remarkably useful properties, as we state below. 
Set
$$ 
	\calh^\bullet(\calf)=\{P^\bullet\,|\,P\le{}S\}, 
$$
and let $\calf^\bullet\subseteq\calf$ be the full subcategory with object set $\calh^\bullet(\calf)$.

The following is a summary of some important properties of this construction. 
The statements and proofs are contained in \cite[Lem. 3.2, Prop. 3.3, Cors. 3.4-5]{BLO3}

\begin{Prop} \label{bullet props}
The following hold for every saturated fusion system $\calf$ over a discrete 
$p$-toral group $S$.
\begin{enumerate}
\item 
\label{bullet props:Sccs}
The set $\calh^\bullet(\calf)$ contains finitely many $S$-conjugacy classes of subgroups of $S$.

\item 
\label{bullet props:functor}
For subgroups $P,Q\le{}S$, any morphism $f\in\homf(P,Q)$ extends to a unique morphism $f^\bullet\in\homf(P^\bullet,Q^\bullet)$.
Thus  $(-)^\bullet \colon \calf \to{} \calf^\bullet$ is a functor, which is left adjoint to the natural inclusion $\calf^\bullet\subseteq\calf$ and whose unit of adjunction is the inclusion $P \leq P^\bullet$.
Moreover, $(-)^\bullet$ is an idempotent functor, namely $(P^\bullet)^\bullet = P^\bullet$, and it carries inclusions to inclusions, i.e., if $P \leq Q$ then $P^\bullet \leq Q^\bullet$.
\end{enumerate}
\end{Prop}

\begin{Defi}\label{def-fusion-preserving}
Let $\calf$ be a fusion system over $S$. An automorphism $\phi \colon S \to S$ is called \emph{fusion preserving} if for every morphism $f \in \morf{P}{Q}$ there exists a morphism $f' \in \morf{\phi(P)}{\phi(Q)}$ such that $\phi \circ f = f' \circ \phi$.
\end{Defi}

If $\calf$ is a fusion system over $S$, and $\phi$ is a fusion preserving automorphism of $S$, then
 $\phi$ induces an automorphism $\phi_* \colon \F \to \F$ where $\phi_*(P)=\phi(P)$ for all $P \leq S$ and $\phi_*(f)=\phi \circ f \circ \phi^{-1}$ for all morphisms $f \in \F$. Such automorphisms of $\calf$ are referred to  in the literature as ``isotypical". 

\begin{Defi}\label{def-functor-covers-F}
Let $\calg = \SFL$ be a $p$-local compact group. Let $\phi \colon S\to S$ be a fusion preserving automorphism.
We say that a functor $\Phi \colon \call' \to \call''$ where $\call'$ and $\call''$ are subcategories of $\call$, \emph{covers} $\phi$ if 
\begin{enumerate}
\item $\pi|_{\call''} \circ \Phi = \phi_* \circ \pi|_{\call'}$ where $\pi\colon\call\to\calf$ is the projection functor, and $\phi_*$ is the isotypical automorphism of $\F$ induced by $\phi$.
\item For each $P,Q\in\call'$ and $g\in N_S(P,Q)$, $\Phi(\widehat{g}) = \widehat{\phi(g)}$.
\end{enumerate}
\end{Defi}

For a $p$-local compact group $(S,\calf,\call)$
let $\calf^{c\bullet}$ denote the full subcategory of $\calf^\bullet$ consisting of the $\F$-centric objects $P \leq S$.
Let $\call^\bullet$ be the full subcategory of $\call$ whose objects are those  of $\calf^{c\bullet}$.

\begin{Prop}\label{prop-lift-bullet-to-l}
There is a unique functor $\call \xto{(-)^\bullet} \call^\bullet \subseteq \L$ which lifts $\calf \xto{(-)^\bullet} \calf^\bullet$ and which satisfies $\iota_Q^{Q^\bullet} \circ \vp = \vp^\bullet \circ \iota_P^{P^\bullet}$ for every morphism $\vp \in \morl{P}{Q}$.
Moreover, $(\widehat{g})^\bullet=\widehat{g}$ for all $g \in N_S(P,Q)$.
\end{Prop}

\begin{proof}
On objects define $\call \xto{(-)^\bullet} \call^\bullet$ by $P \rMapsto P^\bullet$.
Let $\vp \in \morl{P}{Q}$ be any morphism.
By Proposition \ref{bullet props}\ref{bullet props:functor}, $\incl_Q^{Q^\bullet} \circ [\vp]=[\vp]^\bullet \circ \incl_P^{P^\bullet} $, and Lemma \ref{lem-factor-in-L}\ref{lem-factor-in-L:epi} implies that there exists a unique morphism $\vp^\bullet \in \morl{P^\bullet}{Q^\bullet}$ such that
\begin{equation}\label{prop-lift-bullet-to-l:eq-def-bullet}
\vp^\bullet \circ \iota_P^{P^\bullet} = \iota_Q^{Q^\bullet} \circ \vp.
\end{equation}
Define $(-)^\bullet$ on $\vp$ in this way.
The fact that $[\vp^\bullet]=[\vp]^\bullet$ follows from the uniqueness statement in Proposition \ref{bullet props}\ref{bullet props:functor} and shows in particular that $(-)^\bullet$, as defined on $\call$,  lifts $(-)^\bullet \colon \F \to \F^\bullet$, provided we show it is a functor.
The uniqueness of $\vp^\bullet$ solving \eqref{prop-lift-bullet-to-l:eq-def-bullet} easily implies that $(1_P)^\bullet=1_{P^\bullet}$, and that $(-)^\bullet$ respects compositions.
This shows that  $(-)^\bullet$ is a functor on $\L$. The uniqueness of $\vp^\bullet$ in Equation ({prop-lift-bullet-to-l:eq-def-bullet}) also shows that $(-)^\bullet$ with these properties is unique.

To show that $(\widehat{g})^\bullet=\widehat{g}$ for all $g \in N_S(P,Q)$, first note that by Equation (\ref{prop-lift-bullet-to-l:eq-def-bullet}) $(-)^\bullet$ induces the identity on $\Aut_\call(S)$. Then apply $(-)^\bullet$ to the equation $\iota_Q^S\circ\widehat{g} = \widehat{g}\circ\iota_P^S$, where $\widehat{g}$ on the right hand side is the corresponding element of $\Aut_\call(S)$, and use the fact that $\iota_{Q^\bullet}^S$ is a monomorphism
by Corollary \ref{cor-morphisms-of-L-are-mono-and-epi}.
\end{proof}

\begin{Prop}\label{fusion-preserving-respect-bullet}
If $\phi \colon S \to S$ is a fusion preserving automorphism then $\phi(P^\bullet)=\phi(P)^\bullet$ for any $P \leq S$.
\end{Prop}

\begin{proof}
Recall the definition of $P^\bullet$ \cite[Definition 3.1]{BLO3},  $P^\bullet \defeq P\cdot I(P^{[m]})$, where $P^{[m]} = \{g^{p^m}\,|\,g\in P\}\le S_0$, with $p^m=|S/S_0|$, and for any $Q\le S_0$, $I(Q)\defeq S_0^{C_W(Q)}$ is again a subgroup of $S_0$. Clearly $\phi(P^{[m]})=\phi(P)^{[m]}$ for any $P \leq S$, and for $Q\le S_0$  \[I(\phi(Q))=S_0^{C_W(\phi(Q))} = S_0^{\phi C_W(Q)) \phi^{-1}}=\phi(S_0^{C_W(Q)})=\phi(I(Q)).\] 
Therefore, $\phi(P)^\bullet = \phi(P)\cdot I(\phi(P^{[m]})) = \phi(P \cdot I(P^{[m]}))=\phi(P^\bullet)$.
\end{proof}

\begin{Prop}\label{extend-fusion-preserving-from-lbullet-to-l}
Fix a $p$-local compact group $(S,\calf,\call)$ and let $\phi \colon S \to S$ be a fusion preserving isomorphism.
Let $\Phi' \colon \call^\bullet \to \call^\bullet$ be a functor which covers $\phi$.
Then $\Phi'$ extends to a unique functor $\Phi \colon \call \to \call$ which covers $\phi$.
\end{Prop}

\begin{proof}
Define $\Phi$ on objects by $\Phi(P)=\phi(P)$. To define $\Phi$ on morphisms, observe first that
\[[\Phi'(\vp^\bullet)] =\phi_*([\vp^\bullet])=\phi_*([\vp]^\bullet),\quad\text{and that}\quad [\vp]^\bullet \circ \incl_P^{P^\bullet} = \incl_Q^{Q^\bullet} \circ [\vp].\]
Applying $\phi_*$ to the second equation, and using the first, we obtain \[[\Phi'(\vp^\bullet)] \circ \incl_{\phi(P)}^{\phi(P^\bullet)}=\incl_{\phi(Q)}^{\phi(Q^\bullet)} \circ \phi_*([\vp]).\]  Lemma \ref{lem-factor-in-L}\ref{lem-factor-in-L:mon} now implies that there is a unique morphism $\Phi(\vp) \in \morl{\phi(P)}{\phi(Q)}$ such that 
\[
\Phi'(\vp^\bullet) \circ \iota_{\phi(P)}^{\phi(P^\bullet)} = \iota_{\phi(Q)}^{\phi(Q^\bullet)} \circ \Phi(\vp) \qquad  \text{and} \qquad [\Phi(\vp)]=\phi_*([\vp]).
\]
Functoriality of $\Phi$ follows from the uniqueness statement of Lemma \ref{lem-factor-in-L}\ref{lem-factor-in-L:mon}.
By construction, and since $(-)^\bullet$ is the identity on $\call^\bullet$, the functor $\Phi$ extends $\Phi'$.
From Remark \ref{remark-hats} and the last statement of Proposition \ref{prop-lift-bullet-to-l} one deduces that $\Phi(\widehat{g})=\widehat{\phi(g)}$ for all $g \in N_S(P,Q)$, and so $\Phi$  covers $\phi$ as well.
\end{proof}

\newsect{Adams automorphisms of discrete $p$-toral groups}
\label{sec:adams-automorphisms}

In this section we define a certain class of automorphisms of discrete $p$-toral groups, which we call \emph{Adams automorphisms}. These automorphisms are the starting point of our study of unstable Adams operations on $p$-local compact groups. 

Throughout this section, for any discrete group $\Gamma$,  $\Z[\Gamma]$-modules will be referred to as $\Gamma$-modules.
For a group $G$ which is an extension $0 \to M \to G \to \Ga \to 1$, of a group $\Ga$ by a $\Ga$-module $M$, let  $[G] \in H^2(\Ga,M)$ denote the corresponding extension class (see e.g. \cite[Chapter IV.3]{Br}).

\begin{Defi}\label{Def-aut-extensions}
Let $G$ be a fixed extension of a group $\Ga$ by a $\Ga$-module $M$.
For $\la \in \Aut(M)$ and $\vp \in\ \Aut(\Ga)$, let $\Aut(G;\la,\vp)$ denote the set of all automorphisms $\psi \in \Aut(G)$ which render the following diagram commutative.
\begin{equation}\label{ext-diag}
\xymatrix{
0 \ar[r] &
M \ar[d]^{\la} \ar[r] &
G \ar[d]^{\psi} \ar[r] &
\Gamma \ar[r] \ar[d]^{\vp} &
1 
\\
0 \ar[r] &
M \ar[r] &
G \ar[r] &
\Gamma \ar[r]  &
1
}
\end{equation}

For any group $G$ and subgroups $H,K\le G$, let $\Aut_H(K)\le\Aut(K)$ denote the subgroup of all automorphisms of $K$ which are induced by conjugation by elements of $H$. In particular, if $G=H$ one has $\Aut_G(K) = N_G(K)/C_G(K)$. On the other hand, if $G=K$, then $\Aut_H(G)$ is the image of $H$ in $\Inn(G)$. For an extension as above, notice that $\Aut_M(G) \leq \Aut(G;1_M,1_\Ga)$.
\end{Defi}

Let $M$ be a $\Ga$-module, and let $\varphi\in\Aut(\Ga)$.  Let $\vp^*M$ denote the $\Ga$-module $M$ where the action of $\Ga$ is twisted by $\vp$, namely  $x\in\Ga$ and $m\in M$, $x\cdot m \defeq \vp(x)\cdot m$.  Thus one has an induced homomorphism
 $H^*(\Gamma,M) \xto{\vp^*} H^*(\Gamma,\vp^*M)$.
If $\lambda\in\Aut(M)$ is an automorphism, we say that $\lambda$ is compatible with $\vp$  if  for each $x\in\Ga$ and $m\in M$, $\lambda(x\cdot m) = \vp(x)\cdot\la(m)$. In that case one has an induced homomorphism   $\lambda_* \colon H^*(\Gamma,M) \to H^*(\Gamma, \vp^*M)$.

\begin{Lem}\label{Adams-autos-class}
Let $G$ be an extension of a group $\Gamma$ by a $\Gamma$ module $M$. Fix $\vp \in \Aut(\Gamma)$ and $\lambda \in \Aut(M)$.
Then
\begin{enumerate}
\item 
\label{Adams-autos-class:part1}
$\Aut(G;\la,\vp)$ is not empty if and only if $\lambda$ is compatible with $\vp$, and $\vp^*([G])=\lambda_*([G])$.

\item 
There is an isomorphism $H^1(\Gamma,M)\cong\Aut(G;1_M,1_\Ga)/\Aut_M(G)$, and  so $H^1(\Gamma,M)$ acts freely and transitively on $\Aut(G;\la,\vp)/\Aut_M(G)$.
\label{Adams-autos-class:part2}

\item 
Assume the action of $\Ga$ on $M$ is faithful (i.e., $C_G(M)=M$), and that $\lambda$ is central in $\Aut(M)$.
Then a necessary condition for $\Aut(G;\la,\vp)$ to be nonempty is $\vp=1_{\Ga}$.
\label{Adams-autos-class:part3}

\end{enumerate}
\end{Lem}

\begin{proof}
\ref{Adams-autos-class:part1} 
The classes $\vp^*([G])$ and $\la_*([G])$ correspond to the extensions described in \cite[Definitions 1.3 and 1.5]{BT}, namely the ``backward'' and ``forward'' induced extensions.
The result follows from \cite[Theorem 1.10]{BT}.

\ref{Adams-autos-class:part2} 
Clearly $\Aut(G;1_M;1_\Gamma)$ acts freely and transitively on $\Aut(G;\la,\vp)$ via composition.
By \cite[Theorem 1.10]{BT} there is an isomorphism $\Aut(G;1_M,1_\Ga) \cong \Der(\Ga,M)$ which maps $\Aut_M(G)$ to $\PDer(\Ga,M)$.
Finally, $H^1(\Ga,M) \cong \Der(\Ga,M)/\PDer(\Ga,M)$, see e.g. \cite[Theorem 6.4.5]{We}.

\ref{Adams-autos-class:part3}
By point \ref{Adams-autos-class:part1},  $\Aut(G;\la,\vp)\neq\emptyset$ implies that $\lambda$ is compatible with $\varphi$, and since $\lambda$ is assume to be a central automorphism of $M$, it commutes with the action of $\Gamma.$ Thus, for each $x\in \Ga$ and $m\in M$, 
\[x\cdot\lambda(m) = \lambda(x\cdot m) = \vp(x)\cdot\la(m).\] 
Since the action is assumed faithful, this implies at once that $\vp(x)=x$ for all $x\in \Gamma$.
\end{proof}

In the remainder of this section $S$ always denotes a discrete $p$-toral group with maximal torus $S_0$ and group of components $\Ga$ (Definition \ref{dpt-groups}).
Thus, $S$ is an extension of $\Ga$ by the $\Ga$-module $S_0$.
Note that the centre of $\Aut(S_0) \cong \GL_r(\Z_p)$, where $r=\rk(S_0)$, is isomorphic to the group of the $p$-adic units $\Z_p^\times$ which acts on $S_0$ by power maps, $t \mapsto t^\zeta$, where $\zeta \in \Z_p^\times$ is a fixed unit, and  $t \in S_0$.
Also note that $S_0$ is a characteristic subgroup of $S$, so we obtain homomorphisms $\Aut(S) \xto{\res} \Aut(S_0)$ and $\Aut(S) \xto{q} \Aut(\Ga)$.

\begin{Defi}\label{def-adams-automorphisms}
An \emph{Adams automorphism} on $S$ of degree $\zeta$, where $\zeta \in \Z_p^\times$, is an automorphism of $S$ whose restriction to $S_0$ is the $\zeta$-power automorphism. Clearly composition of Adams automorphisms of $S$ is an Adams automorphism, where the degree is given by the product of the two individual degrees. Thus the collection of all Adams automorphisms of $S$ forms a group which we denote by $\adams(S)$. Let $\adams_\zeta(S)$ denote the subset of all Adams automorphisms of degree $\zeta$. 

The obvious map $\deg \colon \adams(S) \to \Z_p^\times$, which associates with an Adams automorphism its degree, is a group homomorphism. It is the restriction of the natural map $\Aut(S) \xto{\res} \Aut(S_0)$ to $\adams(S)$.

An Adams automorphism is called \emph{normal} if it induces the identity of $\Gamma$.
Let $\adams(S;1_\Ga)$ denote the subgroup of the normal Adams automorphisms and $\adams_1(S;1_\Ga)$ the subgroup of those of degree $1$.
\end{Defi}

\begin{Rmk}\label{remarks-adams-aut}
For any discrete $p$-toral group $S$, the following hold.
\begin{enumerate}
\item
$\adams_1(S)=\Ker(\deg)$, and $\adams(S;1_\Ga)=\Ker(\adams(S) \to \Aut(\Ga))$, so  both subgroups are normal in $\adams(S)$.
\item $\Aut_{S_0}(S) \nsg \adams_1(S;1_\Ga)$.
In fact $\Aut_{S_0}(S)\nsg\Aut(S)$, becuase it is the maximal torus, hence a characteristic subgroup of the $p$-toral group $\Inn(S)=S/Z(S)$, which in turn is normal in $\Aut(S)$.

\item
\label{remarks-adams-aut:union}
$\adams(S)$ is the union of $\Aut(S;\zeta,\ga)$, for all $\zeta \in \Z_p^\times$ and all $\ga \in \Aut(\Ga)$. Here, by abuse of notation, $\zeta$ means the $\zeta$ power map on $S_0$, 
\end{enumerate}
\end{Rmk}

We will almost always restrict attention to normal Adams automorphisms. The next lemma shows that this is in fact a reasonable restriction. 
For a discrete $p$-toral group $S$, the  maximal torus $S_0$ is  not generally self-centralising, but in most important cases it would be. For instance, if $\calf$ is a saturated fusion system over $S$, and one assumes that every $s\in S$ is $\calf$-conjugate to $S_0$ (which can be taken as a concept analogous to connectivity), then $S_0$ is automatically self-centralising.

\begin{Lem}\label{Lem-centric-torus=>normal}
If $S_0$ is self-centralising in $S$, then every Adams automorphism of $S$ is normal, i.e. $\adams(S)=\adams(S;1_\Ga)$.
\end{Lem}

\begin{proof}
This is immedate from Lemma \ref{Adams-autos-class}\ref{Adams-autos-class:part3} and Remark \ref{remarks-adams-aut}\ref{remarks-adams-aut:union}.
\end{proof}

Our next goal is to examine which $p$-adic units $\zeta$ can appear as the degree of a normal Adams automorphisms. For a discrete $p$-toral group $S$, we call any map of sets, which is a right inverse to the projection $\pi\colon S\to\Ga$, a \emph{section} of $S$. Once a section of $S$ is fixed, any element $s\in S$ has a unique expression in the form $s=t\cdot\sigma(x)$, where $x=\pi(s)$ and $t=s\cdot\sigma(x)^{-1}$.

\begin{Lem}\label{Adams aut => section}
Let $\psi$ be a normal Adams automorphism of $S$ of degree $\zeta \neq 1$.
Then there exists a section $\si$ of $S$ which is fixed by $\psi$, namely $\psi \circ \si = \si$.
In particular $\psi(t \cdot \si(x)) = t^\zeta \cdot \si(x)$ for all $t \in S_0$ and $x \in \Ga$.
\end{Lem}

\begin{proof} 
We need to find an element $y$ in any coset $S_0\cdot x$ which is fixed by $\psi$, namely $\psi(y)=y$.
Since $\psi$ is normal, $\psi(s)s^{-1} \in S_0$ for all $s\in S$.
Fix some $x\in S$. Since $S_0$ is divisible and $\zeta \neq 1$, there exists $t \in S_0$ such that $\psi(x)x^{-1} =t^{1-\zeta}$.
Set $y=t\cdot x$.
Then $y\in S_0\cdot x$, and $\psi(y)=\psi(tx)=\psi(t)\psi(x)=t^\zeta \cdot t^{1-\zeta} x = y$ as needed.
\end{proof}

\begin{Defi}\label{Def-p-adic-val-gamma}
The $p$-adic valuation of $z \in \Z_p^\times$, denoted $\nu_p(z)$, is the largest integer $m \geq 0$ such that $z \in p^m\Z_p$.
For any $m\geq 0$, let $\Ga_m(p)$ be the subgroup of $\Z_p^\times$ consising of all $\zeta$ such that $\nu_p(\zeta-1) \geq m$. Thus $\units{0} = \Z^\times_p$ and for any $m>0$ any $\zeta\in\units{m}$ has the form $\zeta = 1 +ap^m +\cdots$.
\end{Defi}

We note that since $\Ga$ is a finite $p$-group, and since $S_0\cong (\Z/p^\infty)^r$, a standard transfer argument shows that $H^i(\Ga,S_0)$ is a finite abelian $p$-group of exponent bounded above by the order of $\Ga$ for all $i >0$.

\begin{Prop}\label{prop-deg-section}
Let $S$ be a discrete $p$-toral group, and assume that the order of its extension class $[S] \in H^2(\Gamma, S_0)$ is $p^m$.
Then
\begin{enumerate}
\item
\label{deg-section:in-gamma}
If $\psi$ is a normal Adams automorphism, then $\deg(\psi)\in\Ga_m(p)$.

\item
\label{deg-section:split-epi}
The homomorphism $\adams(S;1_\Ga) \xto{\deg} \Ga_m(p)$ is split surjective.
Moreover, there exists a section $\si$ of $S$, such that  for any  $\zeta \in \Ga_m(p)$ the function $\psi_\zeta$ which takes any $s=t\cdot \sigma(x)\in S$ to $t^\zeta\cdot\sigma(x)$ is a normal Adams automorphism of degree $\zeta$ on $S$, and the map taking  $\zeta\in\units{m}$ to $\psi_\zeta$ is a right inverse to the degree map. 

\item
\label{deg-section:H1}
The kernel of the degree homomorphism in \ref{deg-section:split-epi} is $\adams_1(S;1_\Ga)$ and there is an isomrphism $\adams_1(S;1_\Ga)/\Aut_{S_0}(S) \cong H^1(\Ga,S_0)$.
\end{enumerate}
\end{Prop}

\begin{proof}
\ref{deg-section:in-gamma}
Set $\eta=[S]$.
By Lemma \ref{Adams-autos-class}\ref{Adams-autos-class:part1}, $\Aut(S,\zeta,1_{\Gamma})$ is not empty if and only if $\eta=\eta^\zeta$, namely $\eta^{\zeta-1}=0 \in H^2(\Ga,S_0)$.
This happens if and only if $\zeta \in 1+ p^m\Z_p$, namely $\zeta \in \Ga_m(p)$.
In addition this shows that the degree map in \ref{deg-section:split-epi} is surjective.

\ref{deg-section:split-epi}
Fix some $\psi \in \Aut(S;1+p^m;1_{\Gamma})$ which is not empty by \ref{deg-section:in-gamma}.
By Lemma \ref{Adams aut => section} there exists a section $\si$ of $S$ fixed by $\psi$.
We may assume that $\si(1)=1$.
Note that $\delta_{x,y} \defeq \si(x)\si(y)\si(xy)^{-1}$ is fixed by $\psi$ for any $x,y \in \Ga$.
But $\psi(\delta_{x,y})=(\delta_{x,y})^{1+p^m}$ because  $\de_{x,y} \in S_0$.
Thus, $(\delta_{x,y})^{p^m}=1$ for all $x,y \in \Ga$.
Define for every $\zeta \in \Ga_m(p)$ a function $\psi_\zeta \colon S \to S$ by 
\[
\psi_\zeta(t \cdot \si(x)) = t^\zeta \cdot \si(x), \qquad (t \in S_0, x \in \Ga).
\]
It is easy to check that $\psi(1)=1$ because $\si(1)=1$, that $\psi_\zeta(t_1\si(x) \cdot t_2\si(y))=\psi_\zeta((t_1\si(x)) \cdot\psi_\zeta( t_2\si(y))$ because $(\delta_{x,y})^{p^m}=1$, and that $\psi_\zeta$ induces the identity on $\Ga$ and restricts to  the $\zeta$-power map on $S_0$.
In other words $\psi_\zeta$ is a normal Adams automorphism of degree $\zeta$.
Also $\psi_\zeta \circ \psi_{\zeta'}=\psi_{\zeta\zeta'}$ by inspection.
Hence  the map which takes $\zeta\in\units{m}$ to $\psi_\zeta$ is a right inverse to the degree map.

\ref{deg-section:H1}
Apply Lemma \ref{Adams-autos-class}\ref{Adams-autos-class:part2} to $\Aut(S;1_{S_0},1_\Ga)$.
\end{proof}

\newsect{Unsatble Adams Operations.}
\label{sec:unstable-adams-operations}

In this section we revise the definition and some of the known results for unstable Adams operations on compact Lie groups and $p$-compact groups, and suggest a suitable definition for $p$-local compact groups.

For any compact Lie group $G$, a self map $f\colon BG\to BG$ is called an unstable Adams operation of degree $k\in \N$, if $f$ induces multiplication by $k^i$ on $H^{2i}(BG, \Q)$. Notice that for compact Lie groups this is not a $p$-local concept. 

In the seminal paper by Jackowski, McClure and Oliver \cite{JMO}, the authors prove the following.
\begin{Thm}{\cite[Thms. 1, 2]{JMO}}\label{JMO}
Let $G$ 	be a compact connected Lie group. Then
\begin{enumerate}
\item for each natural number $k>0$ there is at most one unstable Adams operation $\psi^k\colon BG\to BG$ of degree $k$. 
\item If $G$ is in addition simple, then there is an isomorphism of monoids
\[\left[BG, BG\right] \cong \Rep(G,G)\wedge\left\{k\geq 0\quad|\quad k=0, \quad \mathrm{or}\quad (k,|W|) = 1\right\},\]
where $W = N_G(T)/T$, and $T$ is the maximal torus of $G$.
\end{enumerate}
\end{Thm}

For $p$-compact groups $(X,BX)$ \cite{DW}, one defines an unstable Adams operation of degree $\zeta$, where $\zeta\in\Z_p$ is a $p$-adic unit, to be any self map $f$ of $BX$ which induces multiplication by $\zeta^i$ on $H^{2i}_{\Q_p}(BX) \defeq H^{2i}(BX,\Z_p)\otimes \mathbb{Q}$. 
In light of \cite[Theorem 9.7]{DW}, this is equivalent to the requirement that the restriction of $f$ to the maximal torus $T$ of $X$ is an Adams automorphism of degree $\zeta$.

\begin{Thm}[\cite{AG,AGMV}]
For any connected $p$-compact group there exists exactly one unstable Adams operation of degree $\zeta$ for every $p$-adic unit $\zeta$.
\end{Thm}

\begin{proof}
This is immediate from \cite[Theorem 1.2]{AG} which shows that $\Out(BX)=\pi_0\Aut(BX)$ is isomorphic to $\Out(D_X)=\Aut(D_X)/W_X$ where $D_X=(W_X,L_X,\{ \Z_p b_\si\})$ is a root datum associated to a $p$-compact group $X$.
The Adams automorphisms form the centre of $\Aut(D_X)$ which is also the the centre of $\Aut(L_X) \cong \GL_r(\Z_p)$ where $r$ is the rank of $X$.
We remark that for odd primes the result can be deduced from \cite[Thm. 1.1]{AGMV}; for $BDI(4)$ defined in \cite{DW1} it is contained in \cite[Thm. 3.5]{N}.
\end{proof}

While it is possible to define unstable Adams operations of $p$-local compact groups by referring to their classifying spaces or their rational cohomology, we choose to make the definition using the combinatorial structure of a $p$-local compact group, in order to enable us to use all the power of the theory in our analysis. It might turn out of course that the definition below is in fact equivalent to the obvious cohomological definition, but this question will not be addressed in this paper.

\begin{Defi}\label{uAop-def}
Let $\calg=(S,\calf,\call)$ be a $p$-local compact group.
An \emph{Adams operation} of degree $\zeta \in \Z_p^\times$ on $\calg$ is a pair $(\psi,\Psi)$ where $\psi$ is a fusion preserving Adams automorphism of $S$ of degree $\zeta$, (see Definitions \ref{def-adams-automorphisms} and \ref{def-fusion-preserving}), and $\Psi \colon \call \to \call$ is an automorphism which covers $\psi$ (see Definition \ref{def-functor-covers-F}).
\end{Defi}

At this stage it is worth pointing out the reason for including condition (ii) in Definition \ref{def-functor-covers-F}. Consider the example of the trivial fusion and linking systems over a discrete torus $T$. Then $\call = \calb(T)$, while in $\calf$ the only morphisms are inclusions. Thus any power map on $T$ induces the identity functor on $\calf$, and therefore the pair $(\zeta, \mathrm{Id}_\call)$ satisfies condition (i) of the definition. However the realization of the identity functor clearly induces the identity on $BT$. The authors are grateful to Alex Gonzalez for this observation. 

Let $\adams(\calg)$ denote the group of all the Adams operations of $\calg$.
Let $\adams(\F)$ denote the group of all the fusion preserving Adams automorphism of $S$.
The assignment $(\Psi,\psi) \mapsto \psi$ gives rise to a homomorphism $\al \colon \adams(\calg) \to \adams(\F) \leq \adams(S)$.
Let   $\adams(\F;1_\Ga)\le \adams(\F)$  be the subgroup of fusion preserving normal Adams automorphisms of $S$. Let $\adams(\calg;1_\Ga)$ denote the group of ``normal'' operations (cf. Definition \ref{def-adams-automorphisms}). Then $\alpha$ restricts to a homomorphism $\alpha\colon\adams(\calg;1_\Ga)\to\adams(\calf,1_\Ga)$. Finally, let $\adams_1(\calg)$ be the subgroup of operations of degree $1$.

We end this section by introducing a more geometric definition of an unstable Adams operation, and prove that the two definitions are equivalent. For a $p$-local compact group $\calg=\SFL$, let $\iota\colon BS\to B\calg$ be the canonical inclusion induced by the distinguished monomorphism $\delta_S\colon S\to \Aut_\call(S)$. Recall \cite[Thm. 6.3]{BLO3} that for any discrete $p$-toral group $Q$ there is a bijection
\[[BQ, B\calg] \cong \Rep(Q,\call)\defeq \Hom(Q,S)/\sim,\]
where for any $\rho, \rho'\in \Hom(Q,S)$,  $\rho\sim\rho'$ if there is some $\chi\in\Hom_\calf(\rho(Q),\rho(Q'))$ such that $\rho' = \chi\circ\rho$. In particular $\Rep(S,\call) = \End(S)/\Aut_\calf(S)$.

\begin{Defi}\label{geo-uAop}
Let $\calg = \SFL$ be a $p$-local compact group. A geometric unstable Adams operation on $\calg$ is a self equivalence  $f$ of $B\calg$ such that  the  homotopy class of the composite \[BS\xrightarrow{\iota} B\calg\xrightarrow{f} B\calg\] represents the class of an Adams automorphism in $\End(S)$ under the identification above. Equivalently $f$ is a geometric Adams operation on $\calg$ if there exists an Adams automorphism $\phi\in\adams(S)$ such that $\iota\circ B\phi \simeq f\circ\iota$.
\end{Defi}

Geometric unstable Adams operations on a $p$-local compact group $\calg$ clearly form a topological monoid, which we denote by $\adams^g(\calg)$, under composition, where the topology is induced from that of the topological monoid of all self homotopy equivalences of $B\calg$.

\begin{Prop}\label{geo-vs-alg-uAops}
For any $p$-local compact group $\calg$, there is an epimorphism of groups $\gamma\colon\adams(\calg)\to\pi_0(\adams^g(\calg))$ whose kernel is the subgroup of all unstable Adams operations $(\psi, \Psi)$ such that $|\Psi|\pcom\simeq Id_{B\calg}$. 
In this case $\psi\in\Aut_\calf(S)$ and $\deg(\psi)$ is a root of unity, whence $\zeta^{p-1}=1$ if $p\neq 2$, and $\zeta=\pm 1$ if $p=2$ (cf. \cite[Sec. 6.7, Prop. 1,2]{Ro}).
\end{Prop}

\begin{proof}
Fix a $p$-local compact group $\calg=\SFL$.
Define $\gamma\colon \adams(\calg)\to \pi_0(\adams^g(\calg))$ by
\[\gamma(\psi, \Psi) = [|\Psi|\pcom]\]
for each  unstable Adams operation $(\psi, \Psi)\in\adams(\calg)$. 
By Condition (ii) of Definition \ref{def-functor-covers-F}, the following square commutes
\[\xymatrix{
BS \ar[rr]^{B\psi}\ar_{\iota}[d] && BS\ar^{\iota}[d]\\
B\calg \ar^{|\Psi|\pcom}[rr]&& B\calg.
}\]
Thus $|\Psi|\pcom$ is a geometric unstable Adams operation on $\calg$, and so $\gamma$ is well defined, sends $(1_S, Id_\call)$ to the identity element and it clearly respects compositions. 

Suppose that $\gamma(\psi, \Psi) = 1\in\pi_0(\adams^g(\calg))$. 
Then $\iota\circ B\psi\sim \iota$, and so by \cite[Thm. 6.3 (a)]{BLO3}, $\psi\in\Aut_\calf(S)$. 
Since $\Out_\calf(S)\defeq \Aut_\calf(S)/\Inn(S)$ is finite by \cite[Def. 2.2, (I)]{BLO3} and since all the elements of $S$ have finite order, this is also the case for all the elements of $\Aut_\F(S)$.
Therefore $\psi^k=\Id$ for some $k >0$.
%
%
But $\psi^k$ induces the $\zeta^k$-power map on the maximal torus $S_0$ and therefore $\zeta^k=1$, namely $\zeta$ is a root of unity.
%
%

It remains to show that $\gamma$ is surjective. We shall  use  \cite[Thm. 7.1]{BLO3} and various points in its proof. Recall first that a self equivalence $\Phi$ of $\call$ is said to be isotypical if for each $P\in\call$, the isomorphism $\Phi_{P,P}\colon \Aut_\call(P)\to \Aut_\call(\Phi(P))$ sends $\delta_P(P)$ to $\delta_{\Phi(P)}(\Phi(P))$.

Let $f\in\adams^g(\calg)$ be a geometric unstable Adams operation. By  \cite[Thm 7.1]{BLO3} there is a group isomorphism
\[\Out\isotyp(\call) \cong
\Out(B\calg),\]
where $\Out\isotyp(\call) = \pi_0(\Aut\isotyp(\call))$ is the group of equivalence classes (under natural isomorphisms) of isotypical self equivalences of $\call$, and $\Out(B\calg)$ is the group of homotopy classes of self equivalences of $B\calg$. Let $\Phi'\in\Aut\isotyp(\call)$ be a representative of the homotopy class of $f$ under this isomorphism, so $|\Phi'|\pcom \simeq f$. Let $\phi'\in\Aut(S)$ denote the restriction of $\Phi'$ to $S\cong\delta_S(S)\le\Aut_\call(S)$. 
Without loss of generality we may assume that $\Phi'$ sends inclusions to inclusions and that $\Phi'(P) = \phi'(P)$ for every $P\in\call$ (cf. \cite[p. 368]{BLO3}).  
This implies that $\Phi'(\widehat{g}) = \widehat{\phi'(g)}$ for all $g\in N_S(P,Q)$ because by applying $\Phi'$ to the equality $\iota_Q^S \circ \widehat{g} = \widehat{g} \circ \iota_P^S$ we obtain $\iota_{\phi'(Q)}^S \circ \Phi'(\widehat{g}) = \widehat{\phi'(g)} \circ \iota_{\phi'(P)}^S$, while from axiom (C) we get $\iota_{\phi'(Q)}^S \circ \widehat{\phi'(g)} = \widehat{\phi'(g)} \circ \iota_{\phi'(P)}^S$, so Corollary \ref{cor-morphisms-of-L-are-mono-and-epi} implies that $\Phi'(\widehat{g})=\widehat{\phi'(g)}$.

The commutative diagram (18) in \cite[p. 370]{BLO3} shows that $\phi'$ is fusion preserving and that $\iota\circ B\phi' \simeq f\circ \iota$.
Furthermore, if $\alpha\in\Iso_\call(P,Q)$, and $g\in P$, then by the fucntoriality of $\Phi'$, and Axiom (C),
\[\widehat{\phi'([\alpha](g))} = \Phi'(\alpha\circ\widehat{g}\circ\alpha^{-1})  = \Phi'(\alpha)\circ\Phi'(\widehat{g})\circ\Phi'(\alpha)^{-1} = \widehat{[\Phi'(\alpha)](\phi'(g))}.\] 
In other words, $\phi'_*([\alpha])=\phi'\circ[\alpha]\circ(\phi')^{-1} = [\Phi'(\alpha)]$. 
It follows that $\Phi'$ covers $\phi'$.

The pair $(\Phi',\phi')$ constructed above satisfies all the requirements from an unstable Adams operation on $\calg$, except, $\phi'$ need not be an Adams automorphism of $S$. 
By the assumption on $f$ and the commutative diagram (18) in \cite[p. 370]{BLO3}, there is some $\alpha\in\Aut_\calf(S)$, such that $\phi' = \phi\circ\alpha^{-1}$, where $\phi$ is an Adams automorphism of $S$, as in Definition \ref{geo-uAop}. 
Choose $\widetilde{\alpha}\in\Aut_\call(S)$ lifting $\alpha$. 
By Lemma \ref{lem-factor-in-L}(i) the morphism $\al|_P \in \Hom_\F(P,\al(P))$ has a unique lift to $\tilde{\al}|_P \in \Mor_\L(P,\al(P))$ such that $\iota_{\al(P)}^S \circ \tilde{\al}|_{P} = \tilde{\al} \circ \iota_P^S$.
Define $\Theta\colon\call\to\call$ by $\Theta(P) = \alpha(P)$, and define $\Theta$ on morphisms by 
\[
\Theta(P\xrightarrow{\varphi} Q) = \widetilde{\alpha}|_Q\circ\varphi\circ\widetilde{\alpha}|_{P}^{-1}.
\]
Then $\Theta$ is clearly well defined and covers $\alpha$ -- condition (i) of Definition \ref{def-functor-covers-F} is immediate and condition (ii) follows from axiom (C). 
In addition the morphisms $\tilde{\al}|_P$ give rise to a natural isomorphism $\Id_{\L} \to \Theta$, hence $|\Theta| \simeq \Id$.
Set $\Phi = \Phi'\circ\Theta$. 
Then $|\Phi'|\pcom\simeq|\Phi'|\pcom\simeq f$, and $\Phi|_{\delta_S(S)} = \phi$. 
Hence $(\phi, \Phi)$ is an unstable Adams operation of $\calg$ such that $\gamma(\phi,\Phi) = [f]$.
\end{proof}

\newsect{A construction of Adams operations}\label{sec:main-construction}

We are now ready to construct unstable Adams operations on $p$-local compact groups. Recall  (Def. \ref{Def-p-adic-val-gamma}) our notataion $\units{m}$ for the group of $p$-adic units $\zeta$ such that $\nu_p(\zeta-1)\ge m$. Throughout this section let $\calg=\SFL$ be a fixed $p$-local compact group.  Fix a set $\objset$ of representatives of the $S$-conjugacy classes of the subgroups in $\calh^\bullet(\F)$. Then $\objset$  is a finite set by \ref{bullet props}\ref{bullet props:Sccs}.
Let $\objset^c\subseteq\objset$ be the subset consisting of those subgroups which are $\F$-centric.

\begin{Thm}\label{main-thm}
Let $\calg=\SFL$ be a $p$-local compact group, let $S_0\le S$ be the maximal torus, and let $\Ga=S/S_0$.
Then
\begin{enumerate}
\item
\label{main-thm:F}
For a sufficiently large integer $k_\calf$ there exists a group homomorphism 
\[
\A_\calf \colon \units{k_\calf} \to \adams(\F;1_\Ga)
\]
such that $\A_\calf(\zeta)$ is a fusion preserving normal Adams automorphism of degree $\zeta$.

\item
\label{main-tm:G}
For a sufficiently large integer $k_\calg$ there exits a group homomorphism
\[
\A_\calg \colon \units{k_\calg} \to \adams(\calg;1_\Ga)
\]
whose composition with $\al \colon \adams(\calg;1_\Ga) \to \adams(\calf;1_\Ga)$ is equal to $\A_\calf|_{\units{k_\calg}}$.
\end{enumerate}
\end{Thm}

\begin{proof}
Let $p^{k_S}$ be the order of the extension class of $S$ in $H^2(\Ga,S_0)$. By Proposition \ref{prop-deg-section}\ref{deg-section:split-epi}, there exists a homomorphism
\[
\A_S \colon \units{k_S} \xto{} \adams(S; 1_\Ga)
\] 
which is a right inverse of the degree map $\deg\colon \adams(S;1_\Ga)\to\units{k_S}$. 
Moreover, there is a section $\si \colon \Ga \to S$ such that $\A_S(\zeta)(t \cdot \si(x))=t^\zeta \cdot \si(x)$ for all $t \in S_0$ and all $x \in \Ga$.

For each $\zeta\in\units{k_S}$ set $\A_S(\zeta)  = \psi_\zeta$. Define a function
\[
\theta \colon S \to S_0, \qquad \theta(x)=x \cdot \si(\bar{x})^{-1},
\]
where $\bar{x}\in \Ga$ is the image of $x\in S$.
Observe that $x=\theta(x) \cdot \si(\bar{x})$ for any $x \in S$ and therefore (see Definition \ref{Def-p-adic-val-gamma})
\begin{equation}\label{main-thm:psi-fixes-x}
\psi_\zeta(x)=x \qquad \text{if $\nu_p(\zeta-1) \geq \log_p(\ord(\theta(x)))$.}
\end{equation}
For any subset $X \subseteq S$, let $\theta(X) \subseteq S_0$  denote the image of $X$ in $S_0$, and for  $Y \subseteq S_0$, denote by $\ord(Y)$ the supremum of the orders of the elements of $Y$.

For every $P \in \objset$ let $\rset_P$ be a set of representatives for the cosets of $P_0$ in $P$.
Given any $P,Q \in \objset$ the set $\morf{P}{Q}/Q$ is finite by \cite[Lemma 2.5]{BLO3}, and hence so is the set $\morf{P}{Q}/N_S(Q)$.
For $P, Q\in\objset$, let $\Mset_{P,Q} \subseteq \F$ be a choice of a set of representatives for $\morf{P}{Q}/N_S(Q)$.
For convenience, if $P \leq Q$, we will always choose $\incl_P^Q$ to represent its image in $\morf{P}{Q}/N_S(Q)$.
In particular $\Id_P$ represents its image in $\Aut_\F(P)/N_S(P)$.

\noindent{\bf Proof of Part \ref{main-thm:F}.}
For any $P,Q \in \objset$ define the following subsets of $S$:
\[
\mset_{P,Q} = \underset{f \in \Mset_{P,Q}}{\cup} f(\rset_P), \qquad \text{ and set } \qquad \mset=\underset{P,Q \in \objset}{\cup} \mset_{P,Q}.
\]
Note that $\rset_P \subseteq \mset$ for any $P \in \objset$,  since $\Id_P \in \Mset_{P,P}$.
Define
\[
k_\F \defeq \max \, \{\,  k_S , \,\log_p \ord(\theta(\mset)) \,\}.
\]
Let $\A_\F \colon \units{k_\F} \to \adams(S;1_\Ga)$ be the restriction of $\A_S$ to $\units{k_\F}$.
We claim that its image is contained in $\adams(\F;1_\Ga)$.
This will complete the proof of \ref{main-thm:F}.

Fix $\zeta \in \units{k_\F}$, and write $\psi=\psi_\zeta$ for short. 
It is clear from Proposition \ref{bullet props}\ref{bullet props:functor} that it is enough to check that $\psi$ is fusion preserving on $\F^\bullet$.
In other words, it suffices to show that $\psi \circ f \circ \psi^{-1} \in \F^\bullet$ for any $P,Q \in \F^\bullet$ and any $f \in \morf{P}{Q}$.

Consider first any $P,Q \in \objset$ and any $f \in \Mset_{P,Q}$.
Since $P_0$ is a discrete $p$-torus, $P_0 \leq S_0$ and also $f(P_0) \subseteq Q_0 \leq S_0$.
Since $\psi|_{S_0}$ is the $\zeta$-power map, it follows that $\psi \circ f \circ \psi^{-1}(x) = f(x)$ for any $x \in P_0$.
For any $x \in \rset_P$, the definition of $k_\F$ implies that $\log_p(\ord(\theta(f(x))) \leq k_\F \leq \nu_p(\zeta-1)$ and therefore Equation \eqref{main-thm:psi-fixes-x} implies that $\psi(f(x))=f(x)$.
Since $P$ is generated by $P_0$ and $\rset_P$, we deduce that $\psi \circ f \circ \psi^{-1} =f$ for any $f \in \Mset_{P,Q}$.
Since $1_P \in \Mset_{P,P}$ by construction, it follows that $\psi(P)=P$ for any $P \in \objset$.

Let $P,Q \in \calh^\bullet(\F)$ be any subgroups, and let $f \in \F^\bullet(P,Q)$ be a morphism.
There  are $P',Q' \in \objset$, $g \in N_S(P,P')$, $h \in N_S(Q',Q)$ and $f' \in \Mset_{P',Q'}$, such that $f = c_h \circ f' \circ c_g$.
Since $\psi \circ f' \circ \psi^{-1} = f'$ one has \[\psi \circ f \circ \psi^{-1} = c_{\psi(h)} \circ f' \circ c_{\psi(g)},\] which is clearly a morphism in $\F$.
It remains to show only that if $P\in\calh^\bullet(\calf)$, then so is $\psi(P)$. Let $P'\in\calp$ be a representative of the $S$ conjugacy class of $P$, and  let $g\in N_S(P',P)$ be any element. Then $\psi(P)=c_{\psi(g)}(P')$, and since  conjugation by an element of $S$ is clearly a fusion preserving automorphism of $S$, the claim follows from Proposition \ref{fusion-preserving-respect-bullet}.
\medskip

\noindent{\bf Proof of Part \ref{main-tm:G}.}
For any morphism $\vp$ in $\L$, denote the corresponding homomorphism in $\F$ by $[\vp]$ as usual.
Let $\L^\bullet$ denote the full subcategory of $\L$ whose object set is $\calh^\bullet(\F^c)$.
For every $P,Q \in \objset^c$ fix a set $\tMset_{P,Q}$ of morphisms in $\morl{P}{Q}$ which are preimages of the homomorphisms in $\Mset_{P,Q}$ under $\pi\colon\call\to\calf$.
In particular take $\iota_P^Q$ for the lift of $\incl_P^Q$ if $P \leq Q$, and note that because of this choice $1_P \in \tMset_{P,P}$ for all $P\in\objset^c$.

Let  $P, Q \in \objset^c$ and let $\vp \in \tMset_{P,Q}$ be any morphism.
Since $\Aut_\L(Q)$ acts freely on $\morl{P}{Q}$ by Corollary \ref{cor-morphisms-of-L-are-mono-and-epi}, and since $\tMset_{P,Q}$ forms a set of representatives for $\morl{P}{Q}/N_S(Q)$, for every $g \in N_S(P)$ there exist unique $\vp' \in \tMset_{P,Q}$ and $\la_{\vp}(g) \in N_S(Q)$ such that
\begin{equation}\label{lambda-phi-def}
\varphi\circ \widehat{g} = \widehat{\lambda_\varphi(g)}\circ\varphi'.
\end{equation}
We thus obtain a function $\la_\vp \colon N_S(P) \to N_S(Q)$, whose useful properties will be exploited later.

Similarly, if $P,Q,R \in \objset^c$ and $\vp \in \tMset_{P,Q}$ and $\ga \in \tMset_{Q,R}$, then there are unique $u_{\gamma,\varphi}\in N_S(R)$, and $\varepsilon\in\tMset_{P,R}$ such that 
\[
\gamma\circ\varphi = \widehat{u_{\gamma,\varphi}}\circ\varepsilon.
\]
In particular, if $Q=R$ and $\gamma=1_Q$, then $\varphi = \widehat{x}$ for some $x\in N_S(P,Q)$ if and only if $\varepsilon = \widehat{y}$ for some $y\in N_S(P,Q)$. 
But $\tMset_{P,Q}$ contains only one morphism in each left $N_S(Q)$ class, and so $x=y$, and $u_{1_Q,\widehat{x}} = 1$. 
Similarly, if $P=Q$, $\varphi = 1_P$ and $x\in N_S(P,R)$, then $u_{\widehat{x},1_P}=1$.

For any $P \in \objset^c$ let $\nset_P$ be a set of representatives for the cosets of $P$ in $N_S(P)$.
We choose $1 \in N_S(P)$ to represent the coset $1 P$.
Consider the following subsets of $S$.
\[
\cset = \bigcup_{P,Q \in \objset^c} \, \bigcup_{\vp \in \tMset_{P,Q}} \la_\vp(\nset_P), \qquad 
\iset = \bigcup_{P,Q \in \objset^c} \, \{ x \in S \;|\; \widehat{x} \in \tMset_{P,Q} \} 
\]\[
\dset = \bigcup_{P,Q,R \in\objset^c} \{ u_{\ga,\vp} \;|\; \vp \in \tMset_{P,Q}, \;\ga \in \tMset_{Q,R} \}. 
\]

Clearly, all these sets are finite. 
Notice also that $\nset_P \subseteq \cset$ for all $P\in \calp^c$, since $1_P \in \tMset_{P,P}$, and since $1_P \circ \widehat{g}=\widehat{g} \circ 1_P$ for all $g \in N_S(P)$, so $\la_{1_P}(g)=g$.
Define
\[
k_\calg = 
\max \{ k_\F, \, \log_p \ord(\theta(\cset)), \, \log_p \ord(\theta(\iset)), \log_p \ord(\theta(\dset)) \, \}.
\]
Fix some $\zeta \in \units{k_\calg}$ and set $\psi=\A_\F(\zeta)$.
Equation \eqref{main-thm:psi-fixes-x} now shows that
\begin{equation}\label{main-thm:psi-fixes-cdin}
\text{$\psi$ fixes every element in the sets } \cset, \iset, \text{ and } \dset.
\end{equation}
Our goal now is to construct $\Psi \colon \L \to \L$ such that $(\psi,\Psi) \in \adams(\calg)$.
We will then show that the resulting assignment $\zeta \mapsto (\psi,\Psi)$ is a group homomorphism $\units{k_\calg} \to \adams(\calg)$. The rest of the proof proceeds in five steps. We first study the properties of the function $\lambda_\vp$ essential to our analysis. In Steps 2  and 3 we define $\Psi$ on $\call^\bullet$ and show that it is a functor. Step 4 is the proof that $\Psi$ covers $\psi$. Finally, in Step 5 we extend the definition of $\Psi$ to $\call$, define the function $\A_\calg$ and show that it is a homomorphism, thus completing the proof of (ii).
\medskip

\noindent
\textbf{Step 1 -- Properties of $\la_\vp$.}
Fix $P, Q \in \objset^c$ and $\vp \in \tMset_{P,Q}$.
Then
\begin{itemize}
\item[(a)]
If $g \in P$ then $\la_\vp(g)=[\vp](g)$.

\item[(b)]
Fix some $g \in N_S(P)$ and  let $\vp' \in \tMset_{P,Q}$ be the unique morphism such that $\vp \circ \widehat{g} = \widehat{\la_\vp(g)} \circ \vp'$.
Then $\psi(g) \in N_S(P)$ and $\vp \circ \widehat{\psi(g)} = \widehat{\psi(\la_\vp(g))} \circ \vp'$.
In particular $\la_\vp(\psi(g)) = \psi(\la_\vp(g))$.
\end{itemize}

Claim (a) follows immediately from axiom (C) of Definition \ref{L-cat} and the definition of $\la_\vp$ in \eqref{lambda-phi-def}.

To prove (b) write $g=hg_0$ for some unique $h\in P$ and $g_0\in\n_P$. 
By Equation \eqref{lambda-phi-def}, $\vp \circ \widehat{g_0} = \widehat{\la_{\vp}(g_0)} \circ \vp'$ for a unique $\vp' \in \tMset_{P,Q}$.
By Axiom (C) in Definition \ref{L-cat}
\[
\varphi\circ\widehat{g} = \varphi\circ\widehat{hg_0} = \varphi\circ\widehat{h}\circ\widehat{g_0} = \widehat{[\varphi](h)}\circ\varphi\circ\widehat{g_0} = \widehat{[\varphi](h)}\circ\widehat{\lambda_\varphi(g_0)}\circ\varphi'.
\]
In particular $\la_\vp(g)=[\vp](h) \cdot \la_\vp(g_0)$ by uniqueness.
In Part \ref{main-thm:F} we have shown that $\psi(P)=P$ for all $P \in \objset$.
Thus in particular $\psi(g) \in N_S(P)$, and $\psi(h) \in P$.
Note that $\psi$ fixes $g_0$ and $\la_\vp(g_0)$ by \eqref{main-thm:psi-fixes-cdin}, since $\curs{n}_P\subseteq\curs{c}$.
Also $[\vp] \in \Mset_{P,Q}$ and by Part \ref{main-thm:F},  $\psi \circ [\vp] \circ \psi^{-1}=[\vp]$.
By applying Axiom (C) again, one has
\begin{multline*}
\vp \circ \widehat{\psi(g)} = \vp \circ \widehat{\psi(h)} \circ \widehat{g_0} = 
\widehat{[\vp](\psi(h))} \circ \varphi\circ\widehat{g_0}=
\widehat{[\vp](\psi(h))} \circ \widehat{\la_\vp(g_0)} \circ \vp' = \\
\widehat{\psi([\vp](h)}) \circ \widehat{\psi(\la_\vp(g_0))} \circ \vp' =
\widehat{\psi(\la_\vp(g))} \circ \vp'.
\end{multline*}
The last claim of (b) follows from Equation (4). 
\medskip

\noindent
\textbf{Step 2 -- Definition of $\Psi \colon \L^\bullet \to \L^\bullet$.}
On objects define $\Psi(P)=\psi(P)$.
This makes sense by Proposition \ref{fusion-preserving-respect-bullet} which states that $\psi(P)$ belongs to $\calh^\bullet(\F)$ if $P$ does.

Let $P,Q \in \calh^\bullet(\F^c)$ be any subgroups and let $\vp \in \morl{P}{Q}$ be a morphism.
There are unique $P',Q'\in\objset^c$ which are $S$-conjugate to $P$ and $Q$ respectively. Choose $g\in N_S(P',P)$ and $h\in N_S(Q,Q')$. In particular if $P=P'$, choose $g=1$, and similarly if $Q=Q'$. 
Then there are unique $\varphi'\in \tMset_{P',Q'}$ and $x\in N_S(Q')$ such that
\begin{equation}\label{main-thm:eq-factor-morphisms}
\widehat{h} \circ \varphi \circ \widehat{g} = \widehat{x} \circ \varphi'.
\end{equation}
Define $\Psi(\varphi)$ to be the composite
\[
\Psi(P) \xrightarrow{\widehat{\psi(g^{-1})}} P' \xrightarrow{\varphi'}Q' \xrightarrow{\widehat{\psi(x)}} \Psi(Q') \xrightarrow{\widehat{\psi(h^{-1})}} \Psi(Q).
\]
In particular, if $P,Q \in \objset^c$ and $\vp \in \tMset_{P,Q}$, then $\Psi(\vp)=\vp$.

We need to show that $\Psi(\vp)$ does not depend on the choice of $g$ and $h$.
If $u\in N_S(P',P)$ and $v\in N_S(Q,Q')$ are any other elements, then there are unique $\varphi''\in \tMset_{P',Q'}$ and $y\in N_S(Q')$ such that $\widehat{v}\circ\varphi\circ\widehat{u} = \widehat{y}\circ\varphi''$. 
For the two corresponding definitions of $\Psi(\varphi)$ to coincide it suffices to show that the diagram
\[
\xymatrix{
\Psi(P) \ar[rr]^{\widehat{\psi(g^{-1})}} \ar@{=}[d] & & 
P' \ar[rr]^{\varphi'} \ar[d]_{\widehat{\psi(u^{-1}g)}} & &
Q' \ar[rr]^{\widehat{\psi(x)}} \ar[d]_{\widehat{\psi(y^{-1}vh^{-1}x)}}  & &
\Psi(Q')\ar[rr]^{\widehat{\psi(h^{-1})}} \ar[d]_{\widehat{\psi(vh^{-1})}} & &
\Psi(Q) \ar@{=}[d] \\
\Psi(P) \ar[rr]_{\widehat{\psi(u^{-1})}} & &
P' \ar[rr]_{\varphi''} & &
Q' \ar[rr]_{\widehat{\psi(y)}} & &
\Psi(Q') \ar[rr]_{\widehat{\psi(v^{-1})}} & & 
\Psi(Q)
}
\]
commutes. By Remark \ref{remark-hats}, the first, third and fourth squares commute.
Therefore, the only square whose commutativity is not obvious is the second, namely
we must show that 
\begin{equation}\label{welldef}
\varphi''\circ\widehat{\psi(u^{-1}g)} = \widehat{\psi(y^{-1}vh^{-1}x)}\circ\varphi'.
\end{equation}

Write $s=u^{-1}g$ and $t = y^{-1}vh^{-1}x$ for short. Thus we must show that $\varphi''\circ\widehat{\psi(s)} = \widehat{\psi(t)}\circ\varphi'$. 
By the choices made, and the definition of $\lambda_{\vp''}$ we have
\begin{equation}\label{main-thm:eq-step1a}
\widehat{t}\circ\varphi'=\varphi''\circ\widehat{s} = \widehat{\lambda_{\varphi''}(s)}\circ\varphi'''.
\end{equation}
Hence $\varphi'=\varphi'''$, and $t= \lambda_{\varphi''}(s)$.
Note that $\psi(s) \in N_S(P')$ because $P' \in \objset$ so  $\psi(P')=P'$.
Similarly $\psi(t) \in N_S(Q')$.
By Step 1 (b), and since $t=\la_{\vp''}(s)$, 
\begin{equation}\label{main-thm:eq-step1b}
\varphi''\circ\widehat{\psi(s)} = \widehat{\lambda_{\varphi''}(\psi(s))}\circ\varphi_0'' = \widehat{\psi(\lambda_{\varphi''}(s))}\circ\varphi_0'' = \widehat{\psi(t)}\circ\varphi_0'',
\end{equation}
for a unique $\varphi_0'' \in \tMset_{P',Q'}$. 
Reducing this equation to $\calf$ and recalling that $\psi_*$ fixes $\Mset_{P',Q'}$, we have 
\[\psi\circ [\varphi'']\circ c_s\circ\psi^{-1}=[\varphi'']\circ c_{\psi(s)} = c_{\psi(t)}\circ[\varphi_0''] = \psi\circ c_t\circ[\varphi_0'']\circ\psi^{-1}.
\]
By applying $\psi_*$ to the reduction of Equation  \eqref{main-thm:eq-step1a} to $\F$, the same equation holds with $[\varphi_0'']$ replaced by $[\varphi']$.
It follows that there exists a unique $z\in Z(P')$, such that $\varphi_0'' = \varphi'\circ\widehat{z}$. 
But, $\varphi'\circ\widehat{z} = \widehat{[\varphi'](z)}\circ\varphi'$ by Axiom (C), and so $\varphi_0'' =  \widehat{[\varphi'](z)}\circ\varphi'$. 
Since both $\varphi_0''$ and $\varphi'$ are in $\tMset_{P',Q'}$ are unique representatives of their class modulo the left action of $N_S(Q')$, we must have $\varphi_0''=\varphi'$ and $z=1$. 
Returning to Equation \eqref{main-thm:eq-step1b} we deduce that $\vp'' \circ \widehat{\psi(s)}=\widehat{\psi(t)} \circ \vp'$ which shows, in turn, that $\Psi$ is well defined on $\call^\bullet$. 

\medskip

\noindent
\textbf{Step 3 -- $\Psi \colon \L^\bullet \to \L^\bullet$ is a functor.}
First notice that if $P\in\call^\bullet$ is any object, and $P'\in\objset^c$ represents its $S$ conjugacy class, then for any $g\in N_S(P',P)$, the composite
\[
\psi(P)\xrightarrow{\widehat{\psi(g)^{-1}}} P' \xrightarrow{1_{P'}} P' \xrightarrow {\widehat{\psi(g)^{}}} \psi(P)
\]
is the identity on $\psi(P)$. 
Hence $\Psi(1_P) = 1_{\Psi(P)}$.

It remains to prove that $\Psi$ repsects compositions.
Fix $P,Q, R\in \calh^\bullet(\F^c)$ and let $P',Q',R' \in \objset^c$ be the unique representatives of their $S$-conjugacy class.
Fix elements $x \in N_S(P',P)$, $y \in N_S(Q',Q)$ and $z \in N_S(R',R)$.
Fix morphisms $\al \in \morl{P}{Q}$ and $\be \in \L(Q,R)$.
There are unique morphisms $\al' \in \tMset_{P',Q'}$ and $\be' \in \tMset_{Q',R'}$ and elements $u \in N_S(Q')$ and $v \in N_S(R')$ such that
\[
\widehat{y^{-1}} \circ \al \circ \widehat{x} = \widehat{u} \circ \al' \qquad \text{and} \qquad
\widehat{z^{-1}} \circ \be \circ \widehat{y} = \widehat{v} \circ \be'.
\]
There is a unique morphism $\be'' \in \tMset_{Q',R'}$ such that $\be' \circ \widehat{u} = \widehat{\la_{\be'}(u)} \circ \be''$.
Also there is a unique morphism $\ga' \in \tMset_{P',R'}$ and a unique element $u_{\be'',\al'} \in N_S(R')$ such that $\be'' \circ \al' = \widehat{u_{\be'',\al'}} \circ \ga'$.
Now,
\[
\widehat{z^{-1}} \circ \be \circ \al \circ \widehat{x} =
\widehat{v} \circ \be' \circ \widehat{u} \circ \al' =
\widehat{v} \circ \widehat{\la_{\be'}(u)} \circ \be'' \circ \al' =
\widehat{v} \circ \widehat{\la_{\be'}(u)} \circ \widehat{u_{\be'',\al'}} \circ \ga'.
\]
Note that $\psi$ fixes $u_{\be'',\al'}$ by \eqref{main-thm:psi-fixes-cdin}.
By definition of $\Psi$ on morphisms in $\L^\bullet$ and Step 1 (b),
\begin{multline*}
\Psi(\al \circ \be) = 
\widehat{\psi(z)} \circ \widehat{\psi(v)}\circ\widehat{\psi(\la_{\be'}(u))}\circ \widehat{\psi(u_{\be'',\al'})} \circ \ga' \circ \widehat{\psi(x)}^{-1} =
\\
\widehat{\psi(z)} \circ \widehat{\psi(v)} \circ \widehat{\cdot \la_{\be'}(\psi(u))} \circ\widehat{u_{\be'',\al'}} \circ \ga' \circ \widehat{\psi(x)}^{-1} =
\widehat{\psi(z)} \circ \widehat{\psi(v)} \circ \widehat{\cdot \la_{\be'}(\psi(u))} \circ \be'' \circ \al' \circ \widehat{\psi(x)}^{-1} =
\\
\widehat{\psi(z)} \circ \widehat{\psi(v)} \circ \be' \circ \widehat{\psi(u)} \circ \al' \circ \widehat{\psi(x)}^{-1} =
\widehat{\psi(z)} \circ \widehat{\psi(v)} \circ \be' \circ \widehat{\psi(y)^{-1}} \circ \widehat{\psi(y)} \circ \widehat{\psi(u)} \circ \al' \circ \widehat{\psi(x)}^{-1} = 
\\
\Psi(\be) \circ \Psi(\al).
\end{multline*}

\medskip

\noindent
\textbf{Step 4 -- $\Psi$ covers $\psi$.} We show next that $\Psi \colon \L^\bullet \to \L^\bullet$ covers $\psi \in\adams_\calf(S)$ (see Definition \ref{def-functor-covers-F}).

Fix $P,Q \in \F^{\bullet c}$. We show first that $\Psi(\widehat{x})=\widehat{\psi(x)}$ for any $x \in N_S(P,Q)$.
Let $P',Q'\in \objset^c$ represent the $S$ conjugacy classes of $P$ and $Q$ respectively, and let $g\in N_S(P', P)$ and $h\in N_S(Q, Q')$ be any elements. 
Then there exist unique $u\in N_S(Q')$ and $\varphi\in\tMset_{P',Q'}$, such that 
\[
\widehat{h}\circ\widehat{x}\circ\widehat{g} = \widehat{u}\circ\varphi.
\]
By Remark \ref{remark-hats}, $\varphi = \widehat{y}$ where $y=u^{-1}hxg \in N_S(P',Q')$.
Thus, $y \in \iset$ and by \eqref{main-thm:psi-fixes-cdin}, $\psi(y)=y$.
Hence, by the definition of $\Psi$,
\[
\Psi(\widehat{x}) = \widehat{\psi(h^{-1}u})\circ\widehat{y}\circ\widehat{\psi(g^{-1})}=\widehat{\psi(h^{-1}uyg^{-1})}= \widehat{\psi(x)}.
\]
It remains to prove that $\pi\circ\Psi = \psi_*\circ\pi$, where $\pi\colon\call^\bullet\to\calf^{\bullet c}$ is the projection.
First, by definition $\Psi(R)=\psi(R)$ for any $R \in \F^{\bullet c}$.
Next, for any morphism $\vp \in \Mor_{\call^\bullet}(P,Q)$,
we have to prove that $\psi_*([\vp])=[\Psi(\vp)]$.
Fix $g,h,x$ and $\vp'$ such that  $\widehat{h}\circ\vp\circ \widehat{g} = \widehat{x}\circ \vp'$, as in \eqref{main-thm:eq-factor-morphisms}, so that by definition $\Psi(\vp)=\widehat{\psi(h^{-1}x)} \circ \vp' \circ \widehat{\psi(g^{-1})}$.
In particular $[\vp]=c_{h^{-1}x} \circ [\vp'] \circ c_{g^{-1}}$.
Note also that $[\vp'] \in \Mset_{P',Q'}$,  as it is the reduction to $\calf$ of $\varphi'\in \tMset_{P',Q'}$, and that by the proof of Part \ref{main-thm:F},  $\psi_*([\vp'])=[\vp']$.
Therefore 
\begin{multline*}
\psi_*([\vp])=\psi \circ (c_{h^{-1}x}\circ [\vp'] \circ c_{g^{-1}}) \circ \psi^{-1} = 
c_{\psi(h^{-1}x)} \circ [\vp'] \circ c_{\psi(g^{-1})} =\\
[\widehat{\psi(h^{-1}u)} \circ \vp' \circ \widehat{\psi(g^{-1})}] = [\Psi(\vp)].
\end{multline*}

\medskip

\noindent
\textbf{Step 5 -- Extension to $\call$ and the homomorphism $\A_\calg$.}
Proposition \ref{extend-fusion-preserving-from-lbullet-to-l} and Step 3 above imply that the functor $\Psi \colon \L^\bullet \to \L^\bullet$ constructed above extends uniquely to a functor $\Psi \colon \L \to \L$ which covers $\psi$.
This defines a function $\A_\calg \colon \units{k_\calg} \to \adams(\calg;1_\Ga)$ whose composition with $\adams(\calg) \to \adams(\calf)$ is clearly the restriction of $\A_\F$ to $\units{k_\calg}$.

It remains to check that $\A_\calg$ is a homomorphism. First, if $\zeta=1$ then $\psi=\A_\F(\zeta)=1_S$, and it is easy to check that the definitions in Step 2 implies that the functor $\Psi$ is the identity on $\L^\bullet$ and therefore extends to the identity on $\L$, by Proposition \ref{extend-fusion-preserving-from-lbullet-to-l}.

Next fix $\zeta_1,\zeta_2 \in \units{k_\calg}$ and set 
\[(\psi_{\zeta_i},\Psi_{\zeta_i}) \defeq \A_\calg(\zeta_i)\; i = 1,2,\quad\text{and}\quad
(\psi_{\zeta_1\zeta_2},\Psi_{\zeta_1\zeta_2})\defeq \A_\calg(\zeta_1\zeta_2).\]
By Part \ref{main-thm:F}, $\psi_{\zeta_1} \circ \psi_{\zeta_2}=\psi_{\zeta_1 \zeta_2}$, and it remains to prove that $\Psi_{\zeta_1} \circ \Psi_{\zeta_2}=\Psi_{\zeta_1 \zeta_2}$.
Alternaively, we need to prove that $\Phi \defeq \Psi_{\zeta_1} \circ  \Psi_{\zeta_2} \circ \Psi_{\zeta_1\zeta_2}^{-1}$ is the identity functor on $\L$.
On objects $P \in \L$ it is immediate that $\Phi(P)=P$, because $\Psi_\zeta(P)=\psi_\zeta(P)$ for any $\zeta$.
Also, it follows from Step 4 applied to $\Psi_{\zeta_1},\Psi_{\zeta_2}$ and $\Psi_{\zeta_1\zeta_2}$ that $\Phi(\widehat{g})=\widehat{g}$ for all $P,Q \in \L^\bullet$ and all $g \in N_S(P,Q)$.
By Proposition \ref{extend-fusion-preserving-from-lbullet-to-l} it suffices to prove that $\Phi$ is the identity on $\L^\bullet$.
By the construction for any $\zeta$ for which it is defined $\Psi_\zeta$ fixes all morphisms in $\tMset_{P,Q}$ for any $P,Q \in \objset^c$, and therefore so does $\Phi$.
Since every morphism $\vp \in \Mor_{\L^\bullet}(P,Q)$ has the form $\widehat{h} \circ \vp' \circ \widehat{g}$ for some $\vp' \in \tMset_{P',Q'}$ and some $g \in N_S(P,P')$ and $h \in N_S(Q',Q)$, it follows at once that $\Phi(\vp)=\Phi(\widehat{h}) \circ \Phi(\vp') \circ \Phi(\widehat{g})=\vp$.  This completes the proof that $\A_\calg$ is a homomorphism and with that the proof of the theorem.
\end{proof}

We end with a few corollaries  of the proof of Theorem \ref{main-thm}.

\begin{Cor}
Fix a $p$-local compact group $\calg=\SFL$, and let $\zeta\in\units{k_\calg}$ be any element. Let  $(\psi,\Psi) = \A_\calg(\zeta)$. Then there is a set $\objset$ of representative of all $S$-conjugacy class in $\calh^\bullet(\F)$, and 
\begin{itemize}
\item for each $P,Q \in \objset$ there is a set $\Mset_{P,Q}$ of morphism in $\morf{P}{Q}$, which is a complete set of representatives of $\Hom_\F(P,Q)/N_S(Q)$, and 
\item for all $\F$-centric subgroups $P,Q$  in $\objset$, there are sets $\tMset_{P,Q}$ of morphisms in $\Mor_{\L}(P,Q)$ which are representatives for the set $\morl{P}{Q}/N_S(Q)$,
\end{itemize}
such that the following hold.
\begin{enumerate}
\item $\psi(P) = P$, for all $P\in \objset$.
\item $\psi_*(\varphi) = \psi^{-1}\circ\varphi\circ\psi =\varphi$, for each $\varphi\in\Mset_{P,Q}$.
\item $\Psi(\vp)=\vp$ for any $\vp \in \tMset_{P,Q}$.
\end{enumerate}
Moreover, for any $\F$-centric $P,Q \leq S$ and any $g \in N_S(P,Q)$, $\Psi(\widehat{g})=\widehat{\psi(g)}$. In particular $\Psi$ is isotypical. 
\end{Cor}

\begin{Cor}
Let  $p^m$ be the order of the extension class of $S$ in $H^2(\Ga,S_0)$, and let $k_\calf$ and $k_\calg$ be as in Theorem \ref{main-thm}.
Then $k_\calg \geq k_\F \geq m$.
\end{Cor}

\end{document}